\documentclass[graybox]{svmult}

\usepackage{type1cm}        
\usepackage{makeidx}         
\usepackage{graphicx}        
\usepackage{multicol}        
\usepackage[bottom]{footmisc}

\usepackage{newtxtext}       %
\usepackage{newtxmath}       

\makeindex             

\begin{document}

\title*{Revisitation of a Tartar's result \\on
a semilinear hyperbolic system\\ with null condition}
\author{Roberta Bianchini and Gigliola Staffilani}
\institute{Roberta Bianchini \at Sorbonne Universit\'e, LJLL, UMR CNRS-UPMC 7598, 75252-Paris Cedex 05
(France) \& Consiglio Nazionale delle Ricerche, IAC, via dei Taurini 19, I-00185 Rome (Italy)  \email{r.bianchini@iac.cnr.it}
\and Gigliola Staffilani \at Department of Mathematics, Massachusetts Institute of Technology, 77 Massachusetts Avenue,
Cambridge, MA 02139 \email{gigliola@math.mit.edu}}
%
%
\maketitle

\abstract{We revisit a method introduced by Tartar for proving global well-posedness of a semilinear hyperbolic system with \emph{null} quadratic source in one space dimension. A remarkable point is that, since no dispersion effect is available for 1D hyperbolic systems, Tartar's approach is entirely based on spatial localization and finite speed of propagation.}

\section{Introduction}
\label{sec:1}

We consider the semilinear hyperbolic system with quadratic source term:
\begin{equation}
\label{eq:Tartar}
\begin{aligned}
\partial_t u_i + c_i \partial_x u_i + \sum_{j, k} A_{ijk} u_j u_k=0, \\
u_i(x, 0)=\phi_i(x),
\end{aligned}
\end{equation}

$$\quad x \in \mathbb{R}, \quad t \in I, \quad i=1, \cdots, p, $$

where the coefficients $A_{ij}=A_{ji}$ are symmetric, and satisfy the following condition:

$$(A) \qquad \qquad  A_{ijk}=0 \quad \text{if} \quad c_j=c_k \quad \text{for all} \; i=1, \cdots, p.$$

We now make a connection of Assumption (A) with the \emph{null condition} for the semilinear wave equation, which is presented in \cite{Klainerman1} and will be discussed later on.
To this end, looking at the simplest $2\times 2$ case,
\begin{equation}\label{eq:2x2system}
\begin{aligned}
\partial_t u_1+c_1 \partial_x u_1=\alpha u_1u_2, \\
\partial_t u_2+c_2 \partial_x u_2=\beta u_1u_2, \\
\end{aligned}
\end{equation}
we establish a change of variables by defining $w$ such that
\begin{equation}\label{def:w}
u_1=\partial_t w - c_1 \partial_x w, \quad u_2=\partial_t w - c_2 \partial_x w,
\end{equation}
where $u_1, u_2$ are solutions to (\ref{eq:2x2system}). In terms of the new variable $w$, system (\ref{eq:2x2system}) rewrites as follows:
\begin{align*}
\partial_{tt}w - c_1^2 \partial_{xx} w = \alpha (\partial_t w-c_1\partial_x w) (\partial_t w - c_2 \partial_x w), \\
\partial_{tt}w - c_2^2 \partial_{xx} w = \beta (\partial_t w-c_1\partial_x w) (\partial_t w - c_2 \partial_x w).
\end{align*}
Discarding the trivial solution, we end up with the compatibility condition
$$(C) \qquad \qquad  c_1^2=c_2^2, \quad \alpha=\beta.$$
Combining $(C)$ with $(A)$, we are left with $c_1=-c_2$ and $\alpha=\beta$. By a simple rescaling, this yields a classical example of semilinear wave equation with null condition
\begin{equation}\label{eq:John}
\partial_{tt}w-\partial_{xx} w = (\partial_t w)^2 - (\partial_x w)^2,
\end{equation}
first introduced by John in \cite{John}.
As an important research program in nonlinear partial differential equations, the investigation on the long-time behavior of smooth solutions to dispersive equations started in the Eighties with the seminal papers by Klainerman \cite{Klainerman1} and Christodoulou \cite{C}. A deep historical and mathematical survey on the topic can be found in \cite{Lannes}. A general feature is that the linear dispersive terms of the equation tend to force
the solution to spread and to decay, but the contribution of the nonlinear terms is very different. Since dispersion increases with space dimension, a first class of global existence results has been obtained in dimension $d=4$ by Klainerman \cite{Klainerman0}. As showed by John, \cite{John}, in lower space dimensions the nonlinearity can lead to blow up in finite time for arbitrarily small data. In this case, a precise structure of the nonlinearity, the so-called \emph{null form}, introduced by Klainermann \cite{Klainerman1} and Christodoulou \cite{C}, prevents the formation of singularities. \\
Later, an important contribution to extend the notion of null forms was given by Germain, Masmoudi, Shatah, see \cite{GM,GMS2}. The main idea of this approach is to couple spatial localization via the \emph{space-time resonance method}, using space-weighted estimates, with time oscillations, using normal forms.
We refer to a recent paper by Pusateri and Shatah, \cite{Pusateri}, for a result on the semilinear wave equation with \emph{nonresonant bilinear forms}.\\
In the more general case of systems, quadratic source terms satisfying the null condition for the wave equation are actually equivalent to the compatible forms for hyperbolic systems, \cite{HJ}, which are the ones having the weakly sequential continuity described by compensated compactness, \cite{Tartar1}. An analogous result for hyperbolic systems is indeed proved in \cite{GLZ} in the linear case. \\
Actually, the first contribution in the general case of hyperbolic systems of semilinear equations is due to an unpublished paper by Tartar \cite{Tartar}. In \cite{Tartar}, the author provides results on well-posedness and long-time behavior for a semilinear hyperbolic system with quadratic source term, satisfying a non-crossing condition for the charactericts (Assumption (A)), which is actually equivalent to the null condition for the semilinear wave equation, as showed in \cite{HJ}. The approach developed by Tartar is completely different from both the vector-field method \cite{Klainerman1} and normal forms techniques \cite{GM}. In one space dimension, there is indeed no dispersion effect neither for the wave equation, nor for the semilinear hyperbolic system in (\ref{eq:Tartar}). Tartar's idea is then based on spatial localization and finite speed of propagation: since the characteristic velocities of the interacting waves are different (Assumption (A)), then these waves interact only for a finite time before the separation of the cones of dependency. A similar approach was used by Bianchini and Bressan in \cite{Stefano}. \\
We point out that even in the simplest $2 \times 2$ case, system (\ref{eq:2x2system}) reduces to the semilinear wave equation (\ref{eq:John}) only in the special case of $c_1=-c_2=1, \; \alpha=\beta$. 
Therefore, besides the lack of dispersion, techiques which are built \emph{ad hoc} for the wave equation do not apply to systems (\ref{eq:Tartar})-(\ref{eq:2x2system}) for generic constants $c_i, \, A_{ijk}$, and new ideas are needed in that case. The approach developed in \cite{Tartar} will be presented here. This work is indeed a revisitation of the arguments in \cite{Tartar}, where we attempted to fix the notation, prove some intermediate results, as well as provide an explicit description of the objects and tools. \\
The main result, which is due to Tartar, is stated here.
\begin{theorem}[Global existence for small $L^1$ data]
\label{thm:tartar}
Assume condition $(A)$ on system (\ref{eq:Tartar}). Then, there exists $E_0>0$, $k_1>0, k_2>1$ such that, if the initial data $\phi_i, \, i=1, \cdots, p$ satisfy
$$\phi_i \in L^1(\mathbb{R}), \quad \sum_i \|\phi_i\|_{L^1(\mathbb{R})}=\sum_i \varepsilon_i \le E_0,$$
then system (\ref{eq:Tartar}) admits a unique solution, which, $\forall i=1, \cdots, p,$ satisfies
$$\sum_i \|\partial_t u_i + c_i \partial_x u_i \|_{L^1(\mathbb{R}^+ \times \mathbb{R})} \le k_1 \sum_i \|\phi_i\|_{L^1(\mathbb{R})}.$$
Moreover, let $\bar{\phi}_i \in L^1(\mathbb{R})$ with $\sum_i \|\bar{\phi}_i\|_{L^1(\mathbb{R})} < E_0$, and let $\bar{u}_i$ be the solution to (\ref{eq:Tartar}) with initial data $\bar{\phi}_i$. Then
\begin{align*}
\sup_{t \in \mathbb{R}^+} \sum_i \|u_i - \bar{u}_i\|_{L^1(\mathbb{R})} \le k_2 \sum_i \|\phi_i - \bar{\phi}_i\|_{L^1(\mathbb{R})}.
\end{align*}
\end{theorem}
\section{Proof of the theorem}
\label{sec:2}
In this section, we revisit step by step the argument developed by Tartar in \cite{Tartar} for proving Theorem \ref{thm:tartar}. For the sake of clarity, we consider the case of compactly supported initial data $\phi_i(x),\; i=1, \cdots, p$, whose support is contained in an interval $J=[a, b]$ of the real line. However, the same proof applies to the case of more general initial conditions satisfying the assumptions of Theorem \ref{thm:tartar}, as showed at the end of this work.\\
The proof will be split in different steps.
\subsection*{Step 1 - Definition of the main tools}
We start by defining 
\begin{equation}\label{def:dominio-D}
D:=\{(x, t) \; | \; x-c_i t \in J=[a, b], \text{  for  } i=1, \cdots, p\}.
\end{equation}
Consider now the transport equation
\begin{equation}\label{eq:transport}
\begin{aligned}
& \partial_t v_i + c_i \partial_x v_i=f_i(x, t), \quad (x, t) \in D, \\
& v_i(x, 0)=\phi_i(x), \qquad \qquad x \in J,
\end{aligned}
\end{equation}
for $i=1, \cdots, p$, and $f_i$ smooth enough functions on a space-time domain $D$. The explicit solution is given by
\begin{equation}\label{eq:solution-transport}
v_i(x, t)=\phi_i(x-c_i t)+\int_0^t f_i(x-c_i(t-s), s) \, ds \quad (x, t) \in D.
\end{equation}
For later purposes, we define the following space
\begin{equation}\label{def-Vi}
V_i:=\{v_i \text{ solution to (\ref{eq:transport}) defined on D, with  } f_i \in L^1(D), \; \phi_i \in L^1(J)\},
\end{equation}
equipped with norm
\begin{equation}\label{def:norm}
|||v_i|||_{V_i}=\|f_i\|_{L^1(D)}+\|\phi_i\|_{L^1(J)}.
\end{equation}
Without loss of generality, assume now that all the speeds are positive constant values $c_k>0, \; k=1, \cdots, p$. According to the definition in (\ref{def:dominio-D}), the effective domain $D$ corresponds to the region of the space-time between the lines 
$$x=a+\overline{c}t, \quad x=b+\underline{c}t,$$
intersecting at time $T^*=\dfrac{b-a}{\overline{c}-\underline{c}}$, where
$$\overline{c}:=\max_k c_k, \quad \underline{c}:=\min_k c_k.$$
More precisely, 

\begin{equation}\label{def:explicitD}
D=\{ (x,t) \quad | \quad t \in [0, T^*], \quad  b+\underline{c}t \le x \le a+\overline{c}t \},
\end{equation}

\begin{figure}[h!]
{\def\svgwidth{350pt}
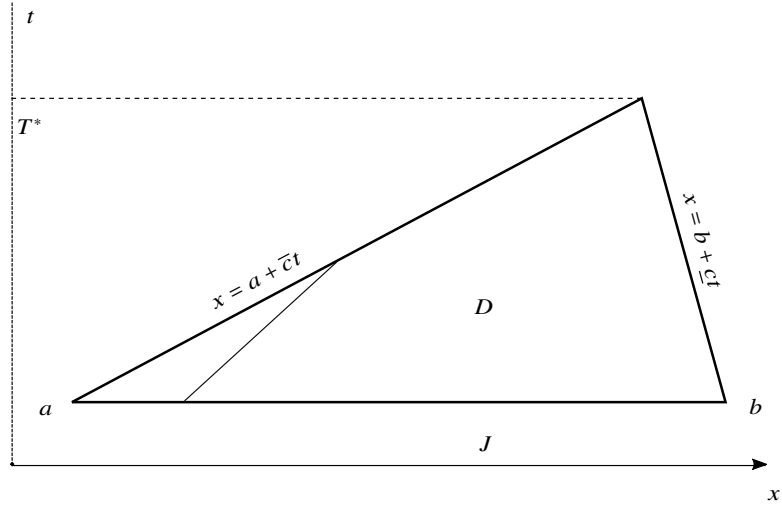}
\caption{Representation of the domain $D$.}
\label{fig:domain}
\end{figure}

as in Figure \ref{fig:domain}.
Now, for any $y \in J=[a, b]$ fixed, consider the following set,
\begin{equation}\label{def:Ky}
K_{y^i}:=\{ \tau \in I \; | \; (y+c_i \tau, \tau) \in D\},
\end{equation}

\begin{figure}[h!]
{\def\svgwidth{350pt}
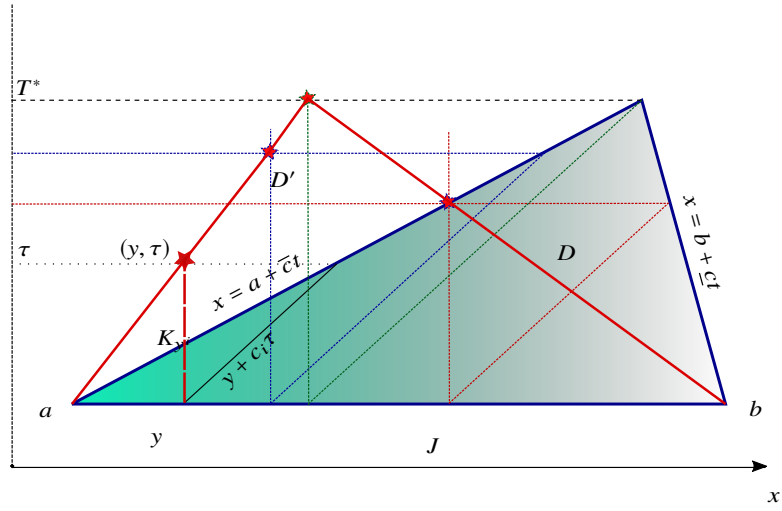}
\caption{Representation of $K_{y^i}$. In this picture, the triangle $\displaystyle D'=\cup_{y \in J } K_{y^i}.$}
\label{fig:ky}
\end{figure}

as in Figure \ref{fig:ky}. For $t < T^*$, we can also define the subset (Figure 3)

\begin{equation}\label{def:Ky}
K_{y_t^i}:=\{ \tau \in I \; | \; (y+c_i \tau, \tau) \in D\} \cap \{0 \le \tau \le t\}.
\end{equation}

\begin{figure}[h!]
{\def\svgwidth{350pt}
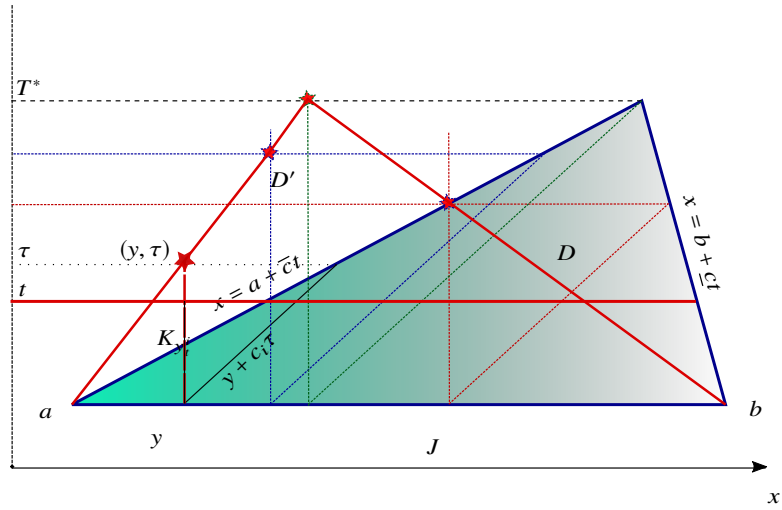}
\caption{Representation of $K_{y_t^i}$, where $\displaystyle D'=\cup_{y \in J } K_{y^i}.$}
\label{fig:kyt}
\end{figure}

Defining the new variable $y=x-c_i t$, formula (\ref{eq:solution-transport}) exactly reads 
\begin{align*}
v_i(y)=\phi_i(y)+\int_{K_{y_t^i}} f_i(y+c_is, s) \, ds, \quad y \in J,
\end{align*}
and then
\begin{equation}\label{eq:norm-Vi-new}
\begin{aligned}
|||v_i|||_{V_i}&=\int_a^b |\phi_i(y)| \, dy + \int_a^b \int_{K_{y^i}} |f_i(y+c_is, s)| \, ds \,  dy \\
& = \int_a^b |\phi_i(y)| \, dy + \int_a^b \int_0^{\min\{ \frac{y-a}{\overline{c}-c_i}, \frac{b-y}{c_i-\underline{c}}\}} |f_i(y+c_is, s)| \, ds \,  dy.
\end{aligned}
\end{equation}
One has the following.
\begin{lemma}
Consider equation (\ref{eq:transport}),whose solution is (\ref{eq:solution-transport}). Then:
\begin{equation}\label{eq:lemma}
\begin{aligned}
\|v_i\|_{L^1(D)} \le c(T^*) |||v_i|||_{V_i}, \quad \text{where} \quad T^*=\dfrac{b-a}{\overline{c}-\underline{c}}.
\end{aligned}
\end{equation}
\end{lemma}
\begin{proof}
\begin{align*}
& \int_{D} |v_i(x,t)| \, dx \, dt  \le \int_D |\phi_i(x-c_i t)| \, dx \, dt + \int_D \int_0^t |f_i(x-c_i(t-s), s)| \, ds \, dt \, dx \\
&\le  \int_0^{\frac{b-a}{\overline{c}-\underline{c}}} \int_{a+\overline{c}t}^{b+\underline{c}t} |\phi_i(x-c_it)| \; dx \, dt + \int_0^{\frac{b-a}{\overline{c}-\underline{c}}} \int_{a+\overline{c}t}^{b+\underline{c}t} \int_0^t |f_i(x-c_it+c_is, s)| \, ds \, dx \, dt\\
& = \int_0^{\frac{b-a}{\overline{c}-\underline{c}}} \int_{a+(\overline{c}-c_i)t}^{b+(\underline{c}-c_i)t} |\phi_i(y)| \; dy \, dt + \int_0^{\frac{b-a}{\overline{c}-\underline{c}}} \int_{a+(\overline{c}-c_i)t}^{b+(\underline{c}-c_i)t} \int_{K_{y_t^i}} |f_i(y+c_is, s)| \, ds \, dy \, dt\\
&  \le  T^* \int_J |\phi_i(y)| \, dy +  \int_0^{\frac{b-a}{\overline{c}-\underline{c}}} dt \int_{a}^{b} \int_{K_{y^i}} |f_i(y+c_i s, s)| \, ds \, dy \\
& \le T^* |||v_i|||_{V_i},
\end{align*}
where the last inequality follows from (\ref{eq:norm-Vi-new}).
\end{proof}
\subsection*{Step 2 - A Fubini-type theorem}
We now perform the same computation for $\displaystyle \int_D |v_i \, v_j| \, dx \, dt$. We will see that space and time play somehow the same role when considering the interactions between waves with different velocities $c_i \neq c_j$, and there will be no dependency on $T^*$ as in inequality (\ref{eq:lemma}).
\begin{proposition}\label{prop-fubini}
If $v_j \in V_j$ and $v_k \in V_k$ with $c_j \neq c_k$ then $v_j v_k \in L^1(D)$ and

\begin{equation}\label{lemma-fubini}
\|v_j v_k\|_{L^1(D)} \le \dfrac{1}{|c_k-c_j|} |||v_j|||_{V_j} |||v_k|||_{V_k}.
\end{equation}
\end{proposition}
\begin{proof}
Notice that the maximal time $t_j$ such that $z+c_j t \in D$ for $t \le t_j$ is determined by the system 
\begin{align*}
z+c_i t = a + \overline{c} t, \quad \text{i.e.} \quad t=\dfrac{z-a}{\overline{c}-c_j}, \\
z+c_i t = b + \underline{c} t, \quad \text{i.e.} \quad t=\dfrac{b-z}{c_j-\underline{c}}. 
\end{align*}
This way, using formula (\ref{eq:norm-Vi-new}), one gets
\begin{align*}
& \int_D  |v_i(x,t) v_j(x,t)| \, dx \, dt 
 \le \frac{1}{|c_i-c_j|} \int_{J \times J} |\phi_i(y) \phi_j(z)| \, dy \, dz\\
& + \frac{1}{|c_i-c_j|}\int_J |\phi_i(y)| \,  dy \, \int_J \int_0^{\min\{ \frac{z-a}{\overline{c}-c_j}, \frac{b-z}{c_j-\underline{c}}\}} |f_j(z+c_js, s)| \, ds \, dz \\
& + \frac{1}{|c_i-c_j|} \int_J |\phi_j(z)| \,  dz \, \int_J \int_0^{\min\{ \frac{y-a}{\overline{c}-c_i}, \frac{b-y}{c_i-\underline{c}}\}} |f_i(y+c_is, s)| \, ds \, dy \\
& + \frac{1}{|c_i-c_j|} \int_{J \times J} dy \, dz \int_0^{\min\{ \frac{y-a}{\overline{c}-c_i}, \frac{b-y}{c_i-\underline{c}}\}} \int_0^{\min\{ \frac{z-a}{\overline{c}-c_j}, \frac{b-z}{c_j-\underline{c}}\}} |f_i(y+c_is, s)| |f_j(z+c_j \tau, \tau)| \, ds \, d\tau \\\\
& \le \frac{1}{|c_i-c_j|} \int_{J \times J} |\phi_i(y) \phi_j(z)| \, dy \, dz\\
& + \frac{1}{|c_i-c_j|}\int_J |\phi_i(y)| \,  dy \, \int_J \int_{K_{z^j}} |f_j(z+c_js, s)| \, ds \, dz \\
& + \frac{1}{|c_i-c_j|} \int_J |\phi_j(z)| \,  dz \, \int_J \int_{K_{y^i}} |f_i(y+c_is, s)| \, ds \, dy \\
& + \frac{1}{|c_i-c_j|} \int_J dy \int_{K_{z^j}} |f_i(y+c_is, s)| \, ds \int_J dz \int_{K_{y^i}}  |f_j(z+c_j \tau, \tau)| \, d\tau, \\
\end{align*}
which ends the proof again from (\ref{eq:norm-Vi-new}).
The following lemma comes directly from (\ref{eq:transport}) and (\ref{def:norm}).
\begin{lemma}\label{lem:equivalence-norms}
Let $v_i$ be the solution to (\ref{eq:transport}) and recall the definition of the norm in (\ref{def:norm}). If $\|\phi_i\|_{L^1(J)} \le \varepsilon $, for some positive constant $\varepsilon$ small enough, then the following holds:

\begin{equation}\label{ineq:equiv-norms}
|||v_i|||_{V_i}-\varepsilon \le \|\partial_t v_i + c_i \partial_x v_i \|_{L^1(D)} \le |||v_i|||_{V_i}.
\end{equation}
\end{lemma}
\subsection*{The fixed point scheme}
We come back to the proof of the theorem defining the iterative scheme
\begin{equation}\label{def:approx-system}
\begin{aligned}
\partial_t v_i^m + c_i \partial_x v_i^m + \sum_{j,k}A_{ijk} v_j^{m-1} v_k^{m-1}=0, \\
v_i^m(x, 0)=\phi_i(x), 
\end{aligned}
\end{equation}
for $m \ge 1$ and $i=1, \cdots, p,$ where: \\
$\bullet$ $\phi_i \in L^1(J)$ are the initial data associated with the Cauchy problem in (\ref{eq:Tartar});\\
$\bullet$ at the first iteration $m=1$, $$v_i^{m-1}=u_i^0:=\phi_i(x-c_i t),$$ i.e. the solution to the linear Cauchy problem:
$$\begin{aligned}
\partial_t u_i^0 + c_i \partial_x u_i^0=0, \\
u_i^0(x, 0)=\phi_i(x).
\end{aligned}$$

Notice that, by definition, $v_i \in V_i$ in (\ref{def-Vi}) at each iteration. We now prove that this iteration model has a fixed point. Denoting by
$$\alpha_i^m:=\|\partial_t v_i^m + c_i \partial_x v_i^m \|_{L^1(D)},$$
by applying Proposition \ref{prop-fubini} to system (\ref{def:approx-system}), one gets:

$$\alpha_i^m \le \sum_{j,k} \dfrac{|A_{ijk}|}{|c_j-c_k|} |||v_j^{m-1}|||_{V_j}|||v_k^{m-1}|||_{V_k}.$$

On the other hand, from Lemma \ref{lem:equivalence-norms},

\begin{align*}
\alpha_i^m \le \sum_{j,k}\dfrac{|A_{ijk}|}{|c_j-c_k|} (\alpha_j^{m-1}+\varepsilon_j)  (\alpha_k^{m-1}+\varepsilon_k),
\end{align*}

where

$$\varepsilon_i:=\|\phi_i\|_{L^1(J)}.$$

Define $\displaystyle \gamma:= \max_{j, k} \sum_i \dfrac{|A_{ijk}|}{|c_j-c_k|}$. Therefore, summing up $i=1, \cdots, p$,

$$\sum_i \alpha_i^m \le \gamma \Bigg(\sum_i \varepsilon_i + \sum_i \alpha_i^{m-1}\Bigg)^2,$$

and if we denote $\displaystyle r_m:=\sum_i \alpha_i^m, \; E_0=\sum_i \varepsilon_i$, one gets

$$r_{m} \le \gamma (E_0+r_{m-1})^2,$$

i.e. the iterative scheme maps the closed set 
$$B_{r_{m-1}}:=\{ u \in V_i, \; i=1, \cdots, p, \, | \, u(x,0)=\phi_i(x), \quad \|\partial_t u + c_i \partial_x u\|_{L^1(D)} \le r_{m-1}\}$$

into $B_{r_m}$,
with $r_m=\gamma (E_0+r_{m-1})^2$.\\
In order to apply the fixed point theorem, we then require that
$$r_m=\gamma (E_0+r_{m-1})^2 \le r_{m-1}.$$
This yields
\begin{equation}\label{ineq:r-constraints}
r:=r_{m-1}   \le \dfrac{1-2 \gamma E_0 + \sqrt{1-4\gamma E_0}}{2 \gamma}, \quad 4 \gamma E_0 < 1.
\end{equation}
We need to check that the iterative scheme defines a strict contraction. We compute
\begin{align*}
\|\partial_t(v_i^m&-v_i^{m-1})+c_i \partial_x (v_i^m-v_i^{m-1})\|_{L^1(D)} \le \sum_{j, k} |A_{ijk}| \|v_j^{m-1} v_k^{m-1} - v_j^{m-2} v_k^{m-2}\|_{L^1(D)}\\
&\le \sum_{j, k} |A_{ijk}| \Bigg( \|v_j^{m-1} (v_k^{m-1}-v_k^{m-2})\|_{L^1(D)} + \|(v_j^{m-1}-v_j^{m-2}) v_k^{m-2}\|_{L^1(D)} \Bigg)\\
& \le \sum_{j, k} \dfrac{|A_{ijk}|}{|c_i-c_j|} \Bigg(|||v_j^{m-1}|||_{V_j} |||v_k^{m-1}-v_k^{m-2}|||_{V_k} + |||v_k^{m-2}|||_{V_k} |||v_j^{m-1}-v_j^{m-2}|||_{V_j}\Bigg)\\
& \le \sum_{k} \dfrac{|A_{ijk}|}{|c_i-c_j|} (r+E_0) |||v_k^{m-1}-v_k^{m-2}|||_{V_k} +  \sum_{j} \dfrac{|A_{ijk}|}{|c_i-c_j|} (r+E_0) |||v_j^{m-1}-v_j^{m-2}|||_{V_j}\\
& = 2 (r+ E_0) \sum_{j} \dfrac{|A_{ijk}|}{|c_i-c_j|}|||v_j^{m-1}-v_j^{m-2}|||_{V_j}.
\end{align*}
Therefore Lemma \ref{lem:equivalence-norms} yields
\begin{align*}
\sum_i |||v_i^m - v_i^{m-1}|||_{V_i} \le 2 \gamma (r+E_0) \sum_j |||v_j^{m-1}-v_j^{m-2}|||_{V_j}.
\end{align*}
Choosing $r=E_0$ (in accordance with (\ref{ineq:r-constraints})), the Lipschitz constant is $4 \gamma E_0$ and the iterative scheme defines a strict contraction on $B_r$ provided that
$$4 \gamma E_0 < 1.$$
\subsection*{Uniqueness}
Now let $u_i^m, \bar{u}_i^m$ be two outputs of the iterative scheme, with initial data $\phi_i, \, \overline{\phi}_i$ respectively. We use the following notation:

\begin{align*}
\|\phi_i\|_{L^1(J)}= \varepsilon_i, \; \sum_i \varepsilon_i=E_0;  \quad  \|\overline{\phi}_i\|_{L^1(J)}= \overline{\varepsilon}_i, \; \sum_i \overline{\varepsilon}_i=\bar{E}_0<E_0.
\end{align*}

Taking the difference, one has
\begin{align*}
\sum_i |||u_i^m - \bar{u}_i^m|||_{V_i} & \le \gamma \sum_{j, k} \Bigg( (\alpha_j+\varepsilon_j) |||u_k^{m-1}-\bar{u}_k^{m-1}|||_{V_k} + (\overline{\alpha}_k+\overline{\varepsilon}_k) |||u_j^{m-1}-\bar{u}_j^{m-1}|||_{V_j}\Bigg)\\
& \le \gamma (2r+E_0+\bar{E}_0) \sum_k |||u_k^{m-1}-\bar{u}_k^{m-1}|||_{V_k}\\
& \le 4 \gamma E_0  \sum_k |||u_k^{m-1}-\bar{u}_k^{m-1}|||_{V_k}\\
& \le \cdots\\
& \le (4\gamma E_0)^m \sum_k \|\phi_i - \overline{\phi}_i\|_{L^1(J)}  ,
\end{align*}
which yields the proof of uniqueness, since $4 \gamma E_0 < 1$.
\subsection*{$L^1$ initial data}
In the case of more general initial data $\phi_i \in L^1(J)$, we consider a partition of the real line $\mathbb{R}=\cup_h J_h$, where $J_h$ are closed connected intervals  with $\sum_h \|\phi_h\|_{L^1(J_h)} \le E_0$. Let $D_h$, defined as (\ref{def:dominio-D}), be the domain of the solution $u_i^h$ related to the interval $J_h$. When $J_{h_1} \cap J_{h_2} \neq \emptyset $, the two solutions $u_i^{h_1}, \, u_i^{h_2}$ coincide thanks to uniqueness proved above. Therefore we can glue the $u_i^h$ together to get our solution.
\end{proof}

\section*{Acknowledgement}
Roberta Bianchini thanks Roberto Natalini for introducing her this line of research and the beautiful paper by Tartar \cite{Tartar}.\\
This project has received funding from the European Research Council (ERC) under the European Union's Horizon 2020 research and innovation program Grant agreement No 637653, project BLOC ``Mathematical Study of Boundary Layers in Oceanic Motion’’. This work was  supported by the SingFlows project, grant ANR-18-CE40-0027 of the French National Research Agency (ANR).

%
%
%

\end{document}